\documentclass[11pt]{amsart}

\usepackage[dvipsnames]{xcolor}

\usepackage[utf8]{inputenc}
\usepackage{amsmath,amsthm,amssymb}
\usepackage{indentfirst}
\usepackage{url}
\usepackage{tikz-cd}
\usepackage[colorlinks]{hyperref}
\usepackage{colonequals}
\usepackage{graphicx}
\usepackage{enumerate}
\usepackage{soul}
\usepackage{stmaryrd}
\usepackage{mathrsfs,mathtools}
\usepackage{palatino}
\usepackage{bbm}
\usepackage[alphabetic]{amsrefs}

\usepackage{ytableau}

\usepackage[margin=0.8in]{geometry}

\newtheorem{thm}[subsection]{Theorem}
\newtheorem{lem}[subsection]{Lemma}

\newtheorem{prop}[subsection]{Proposition}
\newtheorem{cor}[subsection]{Corollary}
{
\theoremstyle{definition}

\newtheorem{defnx}[subsection]{Definition}

}

\newenvironment{defn}
{\pushQED{\qed}\defnx}
{\popQED\enddefnx}

\newcommand{\del}{\partial}

\newcommand{\ZZ}{\mathbb Z}

\newcommand{\wt}{\mathrm{wt}}
\newcommand{\conv}{\mathrm{conv}}

\newcommand{\supp}{\mathrm{supp}}

\definecolor{darkred}{rgb}{0.7,0,0}

\DeclareMathOperator{\img}{img}

\begin{document}
\title{Asymptotically maximal Schubitopes}
\author{Jack Chen-An Chou}
\address{Jack Chen-An Chou, School of Mathematics, University of Minnesota, Minneapolis, MN 55455.}
\email{chou0188@umn.edu}

\author{Linus Setiabrata}
\address{Linus Setiabrata, Department of Mathematics, Massachusetts Institute of Technology, Cambridge, MA 02139.}
\email{setia@mit.edu}

\maketitle
\begin{abstract}
We find a layered permutation $w\in S_n$ whose Schubert polynomial $\mathfrak S_w(x_1, \dots, x_n)$ has support of size asymptotically at least $n!/4^n$. This gives precise asymptotics for the growth rate of $\beta(n)\colonequals \max_{w\in S_n}|\supp(\mathfrak S_w)|$. We find a different layered permutation $w\in S_n$ whose Grothendieck polynomial has support of size asymptotically at least $n!/e^{\sqrt{2n} \cdot \ln(n)}$ and obtain more precise asymptotics for the growth rate of $\beta^{\mathfrak G}(n)\colonequals \max_{w\in S_n}|\supp(\mathfrak G_w)|$.
\end{abstract}
\section{Introduction}
Schubert polynomials $\mathfrak S_w(x_1, \dots, x_n)$, indexed by permutations $w\in S_n$, are lifts of Schubert cycles in the cohomology of the flag variety \cite{ls82}. The specialization $\mathfrak S_w(1,\dots,1)$ is equal to the number of reduced pipe dreams of $w$, and has a geometric interpretation as the degree of the matrix Schubert variety of $w$. Writing $u(n)\colonequals \max_{w\in S_n}\mathfrak S_w(1, \dots,1)$, Stanley \cite{stanley17} observed that
\[
\frac14 \leq \liminf_{n\to\infty} \frac{\log_2(u(n))}{n^2} \leq \limsup_{n\to\infty}\frac{\log_2(u(n))}{n^2} \leq \frac12
\]
and asked whether $\lim_{n\to\infty}\frac{\log_2(u(n))}{n^2}$ exists, and, if so, what the value of this limit is. His question remains open, but see \cite{ms16,mpp19,gao21,gl24,mppy25,zhang25}, and in particular \cite{app26}, for recent progress on this problem and its variants.

We study the growth rate of the maximal sizes of \emph{supports} of Schubert polynomials. Write $\supp(\mathfrak S_w)\colonequals \{\alpha \in \ZZ_{\geq 0}^n\colon \mathbf x^\alpha\textup{ appears in } \mathfrak S_w\}$.
\begin{thm}
\label{thm:main}
Let $\beta(n)\colonequals \max_{w\in S_n}|\supp(\mathfrak S_w)|$. Then
\[
\lim_{n\to\infty}\frac{\ln(\beta(n))}{n\ln(n)} = 1.
\]
More precisely,
\[
-\ln(4)-1 \leq \liminf_{n\to\infty}\frac{\ln(\beta(n)) - n\ln(n)}n \leq \limsup_{n\to\infty}\frac{\ln(\beta(n)) - n\ln(n)}n \leq -1.
\]
\end{thm}
Theorem~\ref{thm:main} answers a problem \cite[Prob 5.5]{gl24} posed by Guo--Lin. We do not know if the limit $\lim_{n\to\infty}\frac{\ln(\beta(n)) - n\ln(n)}n$ exists, nor do we have a conjecture for its value.

The key ingredient in our proof of Theorem~\ref{thm:main} is the observation that a certain \emph{layered} permutation has support of size asymptotically at least $n!/4^n$. Layered permutations are conjectured \cite{app26} to asymptotically maximize $\mathfrak S_w(1, \dots, 1)$. The same paper \cite{app26} disproved an earlier conjecture \cite{ms16} that layered permutations genuinely maximize $\mathfrak S_w(1, \dots, 1)$, exhibiting a counterexample in the symmetric group $S_{17}$. Guo--Lin conjectured \cite[Prob 5.3]{gl24} that there exist permutations simultaneously maximizing $\mathfrak S_w(1, \dots,1)$ and $|\supp(\mathfrak S_w)|$.

Grothendieck polynomials $\mathfrak G_w$ are inhomogeneous deformations of Schubert polynomials; they are generating functions for (possibly nonreduced) pipe dreams of $w$. We are able to produce a layered permutation with support of size asymptotically at least $n!/e^{\sqrt{2n} \cdot \ln(n)}$, i.e., of size $n!$ up to a subexponential factor. In particular, the maximal sizes of supports of Grothendieck polynomials satisfy the following more precise asymptotics.
\begin{thm}
\label{thm:groth}
Let $\beta^{\mathfrak G}(n)\colonequals\max_{w\in S_n}|\supp(\mathfrak G_w)|$. Then
\[
\lim_{n\to\infty}\frac{\ln(\beta^{\mathfrak G}(n)) - n\ln(n)}n = -1.
\]
\end{thm}


\section*{Acknowledgements}
We thank David Anderson, Alejandro Morales, Greta Panova, Alexander Postnikov, Avery St.~Dizier, and Dora Woodruff for enlightening discussions about Schubitopes, maximizers, and the many little problems in the area. J.~Chou was partially supported by NSF Grant DMS-2054423.
\section{Background}
\subsection*{Diagrams}
A {\color{darkred}\emph{diagram}} is a subset $D\subseteq [n]\times[m]$. We view $D = (D_1, \dots, D_m)$ as a subset of an $n\times m$ grid, where $D_j\colonequals \{i \in[n]\colon (i,j)\in D\}$ records the row indices of boxes in column $j$.

\begin{defn}
The {\color{darkred}\emph{Rothe diagram}} $D(w)$ of a permutation $w \in S_n$ is defined to be
\[
D(w) = \{(i,j)\in[n]\times[n]\colon i < w^{-1}(j) \textup{ and } j < w(i)\}.
\]
(See Figure~\ref{fig:31542-rothe}, left.)
\end{defn}
\begin{figure}[ht]
\centering
\includegraphics{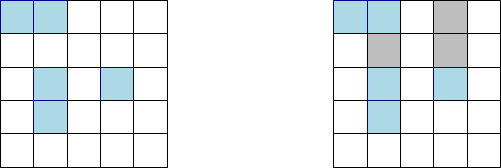}
\caption{The Rothe diagram for $w = 31542$ on the left, and its upwards closure on the right.}
\label{fig:31542-rothe}
\end{figure}

\begin{defn}
Given two sets $A = \{a_1 < a_2 < \dots < a_s\}$ and $B = \{b_1 < b_2 < \dots < b_s\}$ of the same size, we say $A\leq B$ if $a_i \leq b_i$ for all $i$. (If the sets $A$ and $B$ have different size, then they are incomparable.)

Given two diagrams $C = (C_1, \dots, C_n)$ and $D = (D_1, \dots, D_n)$ with the same number of columns, we say $C\leq D$ if $C_i \leq D_i$ for all $i$. (If the diagrams $C$ and $D$ have different numbers of columns, then they are incomparable.)
\end{defn}

\begin{defn}
The {\color{darkred}\emph{weight}} of a diagram $C = (C_1, \dots, C_m)$ is the vector whose $i$-th coordinate is the number $\wt(C)_i = \#\{j \colon i \in C_j\}$ of boxes in the $i$-th row.
\end{defn}

\begin{defn}
Let $D$ be a diagram. The {\color{darkred}\emph{upwards closure}} $\overline D$ of $D$ is
\[
\overline D = \{(i,j)\in[n]\times[n]\colon (i',j)\in D\textup{ for some } i' \geq i\},
\]
i.e., the diagram consisting of squares which are above some square in $D$. (See Figure~\ref{fig:31542-rothe}, right.)
\end{defn}

\subsection*{Schubert and Grothendieck polynomials}
For $i\in[n-1]$, the {\color{darkred}\emph{divided difference operator}} $\del_i\colon \ZZ[x_1, \dots, x_n]\to\ZZ[x_1, \dots, x_n]$ is given by
\[
\del_i(f)\colonequals \frac{f - s_if}{x_i - x_{i+1}},
\]
where $s_i$ is the operator switching the variables $x_i$ and $x_{i+1}$. The {\color{darkred}\emph{isobaric divided difference operator}} $\overline\del_i\colon \ZZ[x_1, \dots, x_n] \to \ZZ[x_1, \dots, x_n]$ is defined to be $\overline\del_i(f)\colonequals \del_i((1-x_{i+1})f)$.

The {\color{darkred}\emph{Schubert polynomial}} $\mathfrak S_w(\mathbf x)$ and {\color{darkred}\emph{Grothendieck polynomial}} $\mathfrak G_w(\mathbf x)$ of $w\in S_n$ are defined by the recursions
\[
\mathfrak S_w(\mathbf x) = \begin{cases} x_1^{n-1}\dots x_{n-1}&\textup{ if } w = w_0,\\\del_i(\mathfrak S_{ws_i}(\mathbf x))&\textup{ if } \ell(w) < \ell(ws_i),\end{cases}\quad\textup{ and } \quad \mathfrak G_w(\mathbf x) = \begin{cases} x_1^{n-1}\dots x_{n-1} &\textup{ if } w = w_0,\\ \overline\del_i(\mathfrak G_{ws_i}(\mathbf x)) &\textup{ if } \ell(w) < \ell(ws_i).\end{cases}
\]
The Schubert polynomial is the lowest degree part of the corresponding Grothendieck polynomial.

The following lemma is an easy consequence of the fact that Schubert and Grothendieck polynomials can be computed using pipe dreams. We give an alternate proof to shorten exposition.
\begin{lem}
\label{lem:upper-bound}
Any monomial $\mathbf x^\alpha$ appearing in a Schubert polynomial $\mathfrak S_w$ or a Grothendieck polynomial $\mathfrak G_w$ divides the staircase monomial $x_1^{n-1} \dots x_{n-1}$.
\end{lem}
\begin{proof}
Since $\mathfrak S_w$ is the lowest degree component of $\mathfrak G_w$, it suffices to show that any monomial appearing in $\mathfrak G_w$ divides $x_1^{n-1}\dots x_{n-1}$. By \cite[Thm 1.2]{mss22}, any monomial appearing in $\mathfrak G_w$ divides $\mathbf x^{\overline{D(w)}}$. Because the $w(i)$-th column of $\overline{D(w)}$ is contained in $\{1, \dots, i-1\}$, the monomial $\mathbf x^{\overline{D(w)}}$ divides $x_1^{n-1}\dots x_{n-1}$.
\end{proof}

The support of a Schubert polynomial can be computed diagrammatically via the following theorem.

\begin{thm}[{\cite[Thm 4]{fms18}, \cite{ary21}}]
\label{thm:schubert-support-via-diagrams}
For any $w\in S_n$, the support of $\mathfrak S_w$ is given by
\[
\supp(\mathfrak S_w) = \{\wt(C)\colon C\leq D(w)\}.
\]
\end{thm}

The Newton polytope of a Schubert polynomial $\conv(\supp(\mathfrak S_w))$ is an example of a \emph{Schubitope}. 

\begin{defn}[\cite{mty19}]
The {\color{darkred}\emph{Schubitope}} $\mathcal S_D$ of a diagram is the convex hull
\[
\mathcal S_D = \conv(\{\wt(C)\colon C\leq D\}).\qedhere
\]
\end{defn}
\begin{lem}[{cf.\ \cite[Thm 7]{fms18}}]
The set $\{\wt(C)\colon C\leq D\}$ is saturated, i.e., 
\[
\mathcal S_D\cap \ZZ^n = \{\wt(C)\colon C\leq D\}.
\]
\end{lem}

\begin{defn}
Given positive integers $b_1, \dots, b_m$, the {\color{darkred}\emph{layered permutation}} with \emph{block sizes} $(b_1, \dots, b_m)$ is the permutation $w(b_1, \dots, b_m)$ with one-line notation
\[
c_1 \, (c_1 - 1)\,\dots \,1\, c_2\, (c_2 - 1)\,\dots\,(c_1 + 1)\, \dots\, c_m\, (c_m - 1)\,\dots\,(c_{m-1}+1),
\]
where $c_i \colonequals b_1 + \dots + b_i$.
\end{defn}
\subsection*{Asymptotics}

We will use the following (crude) estimates for the factorial function, whose proofs we include for completeness.
\begin{lem}
\label{lem:factorial-estimate}
The inequalities
\[
e\left(\frac ne\right)^n \leq n! \leq e\left(\frac{n+1}e\right)^{n+1}
\]
hold for all positive integers $n$.
\end{lem}
\begin{proof}
Estimate $\ln(n!) = \sum_{k=1}^n\ln(k)$ by the integrals
\[
\int_1^n \ln(x)\,dx \leq \sum_{k=1}^n \ln(k) \leq \int_1^{n+1}\ln(x)\,dx.
\]
The definite integrals evaluate to $n\ln(n) - n + 1$ and $(n+1)\ln(n+1) - (n+1) + 1$. The claim follows.
\end{proof}
\begin{lem}
\label{lem:floor-fact-est}
The inequality
\[
\lfloor k\rfloor! \geq \frac1{2k}\left(\frac ke\right)^k
\]
holds for all real numbers $k>1$.
\end{lem}
\begin{proof}
Using Lemma~\ref{lem:factorial-estimate}, compute that
\[
\lfloor k\rfloor! \geq \frac1{k+1} \lceil k\rceil! \geq \frac1{k+1} \cdot \frac{\lceil k\rceil^{\lceil k\rceil}}{e^{\lceil k\rceil-1}} \geq \frac1{k+1}\cdot \frac{k^k}{e^k}.
\]
As $k > 1$, the claim follows.
\end{proof}
\section{A layered permutation with large Schubitope}
Our main technical result is as follows.
\begin{thm}
\label{thm:tech-main}
Let $w = w(b_1, \dots, b_m)$ be a layered permutation in $S_n$. If $b_m \leq n/2$, then
\[
|\supp(\mathfrak S_w)| \geq b_m!\cdot |\supp(\mathfrak S_{w'})|, \qquad \textup{ where } w'\colonequals w(b_1, \dots, b_{m-1})\in S_{n-b_m}.
\]
\end{thm}
Our proof of Theorem~\ref{thm:tech-main} will use the following lemma.
\begin{lem}
\label{lem:last-block}
Let $w\in S_n$ be a layered permutation of the form $w = w(1, \dots, 1, b_m)$ for $b_m \leq n/2$. For any subset $S\subseteq D(w)$, there exists a diagram $C$ satisfying $C\leq D(w)$ and $C\cap D(w) = S$.
\end{lem}
\begin{proof}
Each column $D_j$ of $D(w)$ is of the form $\{n-b_m + 1, \dots, n-c\}$ for some $c \geq 1$. As $|D_j| \leq b_m-1 < n-b_m$, it follows that for any subset $S_j$ of $D_j$ there exists a set $C_j\leq D_j$ with $C_j \cap D_j = S_j$. Applying this column by column to $D(w)$, the claim follows.
\end{proof}
\begin{proof}[Proof of Theorem~\ref{thm:tech-main}]
For each of the $b_m!$ many vectors $\alpha\in\ZZ_{\geq0}^{b_m}$ satisfying $(0,\dots,0) \leq \alpha \leq (b_m - 1, b_m - 2, \dots, 1, 0)$, we construct an embedding $i_\alpha\colon \mathcal S_{D(w')} \hookrightarrow \mathcal S_{D(w)}$ so that $\img(i_\alpha)\cap\img(i_{\alpha'}) = \emptyset$ for all $\alpha\neq\alpha'$.

Write $w_{\mathrm{LB}}$ for the layered permutation $w(1, \dots, 1, b_m) \in S_n$; the Rothe diagram $D(w_{\mathrm{LB}})$ has boxes only in columns $\{n-b_m+1, \dots, n\}$, and is in particular equal to the ``southeasternmost block'' of $D(w)$. Furthermore, 
\[
D(w)_j = \begin{cases} D(w')_j &\textup{ if } j \in \{1, \dots, n-b_m\}\\ D(w_{\mathrm{LB}})_j&\textup{ if } j\in\{n-b_m+1, \dots, n\},\end{cases}
\]
so in particular a diagram $C = (C_1, \dots, C_n)$ satisfies $C\leq D(w)$ if and only if the diagrams $C' \colonequals (C_1, \dots, C_{n-b_m})$ and $C_{\mathrm{LB}} \colonequals (C_{n-b_m+1}, \dots, C_n)$ satisfy $C'\leq D(w')$ and $C_{\mathrm{LB}} \leq D(w_{\mathrm{LB}})$.

For each $\alpha$ satisfying $(0,\dots,0)\leq\alpha\leq (b_m - 1, b_m - 2, \dots, 1,0)$, use Lemma~\ref{lem:last-block} to choose a diagram $C_\alpha \leq D(w_{\mathrm{LB}})$ whose weight satisfies $\wt(C_\alpha)_{n-b_m + j} = \alpha_j$. Then define the map
\begin{align*}
i_\alpha\colon \mathcal S_{D(w')}&\to \mathcal S_{D(w)}\\
\wt(C')&\mapsto\wt((C', C_\alpha));
\end{align*}
the map $i_\alpha$ is well-defined and injective because it is in fact the translation map $\gamma \mapsto \gamma +\wt(C_\alpha)$. Furthermore, any $\gamma' \in \mathcal S_{D(w')}$ satisfies $\gamma'_{n-b_m+j} = 0$ for all $j\geq 0$, so any vector $\gamma\in\img(i_\alpha)$ satisfies $\gamma_{n-b_m+j} = \alpha_j$. In particular, the $\img(i_\alpha)$ are disjoint subsets of $\mathcal S_{D(w)}$.
\end{proof}
\begin{cor}
\label{cor:upshot}
Fix an integer $n\geq 3$ and let $c\colonequals \max\{k \colon \lfloor n/2^k\rfloor \geq 1\}$ and $d \colonequals n - \sum_{k=1}^c \lfloor n/2^k\rfloor$. Let $w$ be the layered permutation with blocks of size
\[
\underbrace{1,\dots,1}_{d\textup{ many}}, \left\lfloor \frac n{2^c}\right\rfloor, \left\lfloor \frac n{2^{c-1}}\right\rfloor, \dots, \left\lfloor\frac n4\right\rfloor,\left\lfloor \frac n2\right\rfloor.
\]
Then 
\[
|\supp(\mathfrak S_w)| \geq \prod_{k=1}^c\left\lfloor \frac n{2^k}\right\rfloor!.
\]
In particular, 
\[
\ln(|\supp(\mathfrak S_w)|) \geq \left(n - \frac1{\ln(2)}\ln(n) - 2\right)\ln(n) - n\ln(4) - (n-2).
\]
\end{cor}
\begin{proof}
From the inequality
\[
\left\lfloor \frac n{2^k}\right\rfloor \leq \frac n{2^k} = \frac12\left(n - \sum_{i=1}^{k-1} \frac n{2^i}\right) \leq \frac12\left(n - \sum_{i=1}^{k-1} \left\lfloor\frac n{2^i}\right\rfloor\right),
\]
Theorem~\ref{thm:tech-main} may be repeatedly applied to the layered permutation $w$ to obtain
\[
|\supp(\mathfrak S_w)| \geq \prod_{k=1}^c\left\lfloor \frac n{2^k}\right\rfloor!.
\]
Lemma~\ref{lem:floor-fact-est} implies
\begin{align*}
\prod_{k=1}^c\left \lfloor \frac n{2^k}\right\rfloor! &\geq \prod_{k=1}^c\left(\frac{2^{k-1}}n \left(\frac n{2^k e}\right)^{\frac n{2^k}}\right) \\&\geq \frac 1{n^c} \cdot \underbrace{\prod_{k=1}^c \left(\frac1{2^k}\right)^{\frac n{2^k}}}_{\geq 2^{-2n}}\cdot \underbrace{\prod_{k=1}^c \left(\frac ne\right)^{\frac n{2^k}}}_{\geq \left(\frac ne\right)^{n-2}}\label{eqn:estimate}\tag{$*$}\\
&\geq \frac{n^{n-c-2}}{4^n\cdot e^{n-2}}\\&\geq \frac{n^{n - \log_2(n) - 2}}{4^n \cdot e^{n-2}}
\end{align*}
where the estimates in~\eqref{eqn:estimate} follow from
\[
\sum_{k=1}^c k\cdot \frac n{2^k} \leq \sum_{k=1}^\infty k \cdot \frac n{2^k} = 2n
\]
and
\[
\frac n2 + \dots + \frac n{2^c} = n - \frac n{2^c} \geq n-2.\qedhere
\]
\end{proof}

\newtheorem*{thm:main}{Theorem~\ref{thm:main}}
\begin{thm:main}
Let $\beta(n)\colonequals \max_{w\in S_n}|\supp(\mathfrak S_w)|$. Then
\[
\lim_{n\to\infty}\frac{\ln(\beta(n))}{n\ln(n)} = 1.
\]
More precisely,
\[
-\ln(4)-1 \leq \liminf_{n\to\infty}\frac{\ln(\beta(n)) - n\ln(n)}n \leq \limsup_{n\to\infty}\frac{\ln(\beta(n)) - n\ln(n)}n \leq -1.
\]
\end{thm:main}
\begin{proof}[Proof of Theorem~\ref{thm:main}]
By Lemmas~\ref{lem:upper-bound} and~\ref{lem:factorial-estimate}, the inequality 
\[
\ln(|\supp(\mathfrak S_w)|) \leq \ln(n!) \leq (n+1)\ln(n+1) - n
\]
holds for all $w\in S_n$. By Corollary~\ref{cor:upshot}, there exists $w\in S_n$ with 
\[
\ln(|\supp(\mathfrak S_w)|) \geq \left(n-\frac1{\ln(2)}\ln(n) - 2\right)\ln(n) - n\ln(4) - (n-2).
\]
The inequalities
\[
-\ln(4)-1 \leq \liminf_{n\to\infty}\frac{\ln(\beta(n)) - n\ln(n)}n \leq \limsup_{n\to\infty}\frac{\ln(\beta(n)) - n\ln(n)}n \leq -1
\]
follow.
\end{proof}

\section{Grothendieck polynomials with larger support}
A permutation is called {\color{darkred}\emph{fireworks}} if initial elements of decreasing runs are increasing. Layered permutations are, in particular, fireworks.

The support of fireworks Grothendieck polynomials is given by the following formula.

\begin{prop}[{\cite[Thm 1.1]{cs25}}]
\label{prop:fireworks-grothendieck}
Let $w\in S_n$ be a fireworks permutation. Then
\[
\supp(\mathfrak G_w) = \bigcup_{\alpha\in\supp(\mathfrak S_w)}[\alpha,\mathrm{wt}(\overline{D(w)})],
\]
where $[\alpha,\gamma]\colonequals \{\beta \in \ZZ^n \colon \alpha_i \leq \beta_i\leq\gamma_i \textup{ for all $i$}\}$ denotes the componentwise comparison interval.
\end{prop}

\begin{prop}
\label{prop:groth-precise}
Fix an integer $n$, let $k$ be the unique integer so that $\binom{k+1}2\leq n<\binom{k+2}2$, and let $b \colonequals n - \sum_{i=1}^k i$. Let $w$ be the layered permutation with blocks of size $1, 2, \dots, k, b$. Then
\[
|\supp(\mathfrak G_w)| \geq \frac{n!}{n^{k+1}}.
\]
In particular,
\[
\ln(|\supp(\mathfrak G_w)|) \geq n\ln(n) - n - (\sqrt{2n}+1)\ln(n)
\]
\end{prop}
\begin{proof}
Let $\mathbf c = \mathrm{wt}(D(w))$ denote the Lehmer code of $w$, and let $\mathbf d = \mathrm{wt}(\overline{D(w)})$. Proposition~\ref{prop:fireworks-grothendieck} implies that $[\mathbf c, \mathbf d] \subseteq \supp(\mathfrak G_w)$, so
\begin{equation}
\label{eqn:crude-grothendieck}
|\supp(\mathfrak G_w)| \geq \prod_{i=1}^n (d_i - c_i + 1).
\end{equation}

For $i\in[k]$, the $i$-th block of $w$ contributes to $\overline{D(w)}\setminus D(w)$ a rectangle consisting of $i-1$ columns with boxes in rows $\{1, 2, \dots, \binom i2\}$, and (when $b\neq 0$) the $(k+1)$-th block of $w$ contributes to $\overline{D(w)}\setminus D(w)$ a rectangle consisting of $b-1$ columns with boxes in rows $\{1, 2,\dots,\binom{k+1}2\}$ (see Figure~\ref{fig:dbar}, left). Hence, $\overline{D(w)}\setminus D(w)$ has $j$ rows of size $j + (j+1) + \dots + (k-1) + (b-1)$ for all $j\leq k-1$ and has (when $b\neq 0$) $k$ rows of size $b-1$ (see Figure~\ref{fig:dbar}, right).

\begin{figure}[ht]
\includegraphics{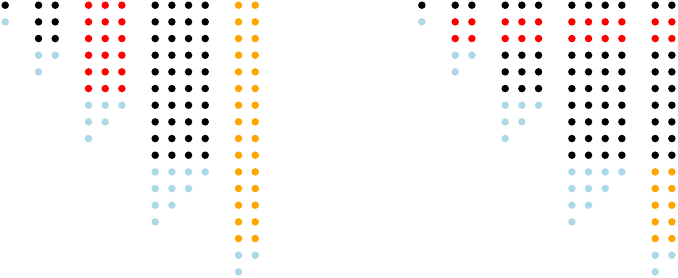}
\caption{Left: In light blue, the Rothe diagram $D(w)$ for $w$ layered with blocks of size $1,2,3,4,5,3$. In red, the fourth block of $w$ contributes to $\overline{D(w)}\setminus D(w)$ a rectangle of $3$ columns with boxes in rows $1$ through $6$. In orange, the last block of $w$ contributes to $\overline{D(w)}\setminus D(w)$ a rectangle of $2$ columns with boxes in rows $1$ through $15$.\\
Right: In light blue, the same Rothe diagram. In red, there are $2$ rows of size $2 + 3 + 4 + 2$, and in orange, there are $6$ rows of size $2$.}
\label{fig:dbar}
\end{figure}Ignoring the contribution of the $k$ rows of size $b-1$, Equation~\eqref{eqn:crude-grothendieck} yields
\begin{align*}
|\supp(\mathfrak G_w)| &\geq \prod_{j=1}^{k-1}\left(n-k-1 - \sum_{i=1}^{j-1}i\right)^j \\&\geq (n-k-1)!\\&\geq\frac{n!}{n^{k+1}}.
\end{align*}
The inequality $\binom{k+1}2\leq n$ gives $k \leq \sqrt{2n}$, and Lemma~\ref{lem:factorial-estimate} gives
\[
\ln(|\supp(\mathfrak G_w)|) \geq \ln\left(\frac{n!}{n^{\sqrt{2n}+1}}\right)\geq n\ln(n) - n - (\sqrt{2n}+1)\ln(n).\qedhere
\]
\end{proof}

\newtheorem*{thm:groth}{Theorem~\ref{thm:groth}}
\begin{thm:groth}
Let $\beta^{\mathfrak G}(n)\colonequals\max_{w\in S_n}|\supp(\mathfrak G_w)|$. Then
\[
\lim_{n\to\infty}\frac{\ln(\beta^{\mathfrak G}(n)) - n\ln(n)}n = -1.
\]
\end{thm:groth}
\begin{proof}[Proof of Theorem~\ref{thm:groth}]
By Lemmas~\ref{lem:upper-bound} and~\ref{lem:factorial-estimate}, the inequality
\[
\ln(|\supp(\mathfrak G_w)|)\leq \ln(n!) \leq (n+1)\ln(n+1)-n
\]
holds for all $w\in S_n$. By Proposition~\ref{prop:groth-precise}, there exists $w\in S_n$ with
\[
\ln(|\supp(\mathfrak G_w)|) \geq n\ln(n) - n - (\sqrt{2n}+1)\ln(n).
\]
The claim follows.
\end{proof}

%

\end{document}